\documentclass[11pt,a4paper,reqno,twoside]{article}
\usepackage{rotating}
\usepackage{hyperref}
\usepackage{amsfonts}
\usepackage{bbding}
\usepackage[all]{xy}
\usepackage{young}
\usepackage{multirow}
\usepackage{booktabs}
\usepackage{amsmath,amscd,amssymb,latexsym,srcltx,indentfirst,titlesec}
\usepackage{amsthm}

\usepackage{bbm}

\usepackage{enumitem}

\setitemize[1]{itemsep=0pt,partopsep=0pt,parsep=\parskip,topsep=5pt}

\numberwithin{equation}{section}
\newtheorem{theorem}{Theorem}[section]
\newtheorem{lemma}[theorem]{Lemma}
\newtheorem{proposition}[theorem]{Proposition}

\newtheorem{remark}[theorem]{Remark}

\newtheorem*{thma*}{Theorem A}
\newtheorem*{thmb*}{Theorem B}
\newtheorem*{thmc*}{Theorem C}
\newtheorem*{thmd*}{Theorem D}
\newtheorem*{thme*}{Theorem E}
\newtheorem{conjecture}[theorem]{Conjecture}
\newtheorem{definition}[theorem]{Definition}

\newtheorem*{thm*}{Main Theorem}
\allowdisplaybreaks

\begin{document}

\makeatletter

\newdimen\bibspace
\setlength\bibspace{2pt}   
\renewenvironment{thebibliography}[1]{%
 \section*{\refname 
       \@mkboth{\MakeUppercase\refname}{\MakeUppercase\refname}}%
     \list{\@biblabel{\@arabic\c@enumiv}}%
          {\settowidth\labelwidth{\@biblabel{#1}}%
           \leftmargin\labelwidth
           \advance\leftmargin\labelsep
           \itemsep\bibspace
           \parsep\z@skip     %
           \@openbib@code
           \usecounter{enumiv}%
           \let\p@enumiv\@empty
           \renewcommand\theenumiv{\@arabic\c@enumiv}}%
     \sloppy\clubpenalty4000\widowpenalty4000%
     \sfcode`\.\@m}
    {\def\@noitemerr
      {\@latex@warning{Empty `thebibliography' environment}}%
     \endlist}

\makeatother

\pdfbookmark[2]{On Oliver's $p$-group conjecture for Sylow subgroups
of unitary groups}{beg}

\title{{\bf On Oliver's $p$-group conjecture for Sylow subgroups
of unitary groups}
\footnotetext{\hspace*{-4 ex}
Department of Mathematics, Hubei University, Wuhan, $430062$, China\\
Supported by NSFC grants (12071117).\\
\hspace*{2 ex}Xingzhong Xu's E-mail: xuxingzhong407@hubu.edu.cn; xuxingzhong407@126.com.
}}

\author{{\small{Xingzhong Xu} } \\
}
\date{}
\maketitle

{\small

\noindent\textbf{Abstract.} {\small{In this paper, we focus on Oliver's $p$-group conjecture. We use elementary method to prove that
Oliver's $p$-group conjecture holds for  Sylow $p$-subgroups
of unitary groups.
 }}\\

\noindent\textbf{Keywords: }{\small { Oliver's $p$-group; the Thompson subgroup. }}

\noindent\textbf{Mathematics Subject Classification (2020):}  \ 20D15.}

\section{\bf Introduction}

In \cite{O}, Oliver proposed Oliver's $p$-group conjecture, a purely group-theoretic problem.
A positive resolution of this conjecture would give the existence and uniqueness of centric linking systems for
fusion systems at odd primes(also see \cite{Ly}), and this can be used to prove the Martino-Priddy conjecture(see \cite{O}).

Curiously, the Martino-Priddy conjecture has been proved by \cite{O, O2}, and the existence and uniqueness of centric linking system
has been solved by \cite{Ch, O3}, but Oliver's $p$-group conjecture is still open.
There are only some works\cite{GHL, GHM, GHM2, Ly, X, LXZ} about this conjecture.

Now, we list Oliver's $p$-group conjecture as follows.

\begin{conjecture}\cite[Conjecture 3.9]{O} Let $S$ be a $p$-group for an odd prime $p$. Then
$$J(S)\leq \mathfrak{X}(S),$$
where $J(S)$ is the Thompson subgroup generated by all elementary abelian $p$-subgroups whose rank is
the $p$-rank of $S$, and $\mathfrak{X}(S)$ is the Oliver subgroup described in \cite[Definition 3.1]{O} and \cite[Definition]{GHL}.
\end{conjecture}

$Motivations~ of~ the~ paper:$ In \cite{GHM2}, they had proved that Conjecture 1.1 holds for the Sylow subgroups of the  symmetric groups and the
general linear groups. That is the one motivation of our paper: we prove that Conjecture 1.1 holds for Sylow subgroups
of other classical groups. The reviewers told us that \cite{O} had prove the related results about these groups.
Here, we use elementary method to prove the following:

\begin{theorem} Let $m\geq1$, let $p\geq 5$ be a prime, and let $q$ be any prime power. Let $S$ be a Sylow $p$-subgroup of the
unitary group $\mathrm{U}_{n}(\mathbb{F}_{q})$(subgroup of $\mathrm{GL}_{n}(\mathbb{F}_{q^2})$). Then Conjecture 1.1 holds for $S$.
\end{theorem}

\begin{proof} First we consider the case where $q$ is a power of $p$ (defining characteristic). By Theorem 3.5 and 3.9, we have
$\mathfrak{X}(S)=S$. So, $J(S)\leq \mathfrak{X}(S)$.

Now we turn to the case where $( p, q)=1$ (coprime characteristic). It was first shown by Weir \cite[p.532]{Weir}
that $S$ is a  wreath product as follows:
$$C_{p^r}\wr C_p\wr\cdots\wr C_p$$ for some integer $r$.
Applying
\cite[Lemma 5.4]{GHM2}, we see that $J(S)$ is elementary abelian by induction. So $J(S)\leq  \mathfrak{X}(S)$.
\end{proof}

$Structure~ of ~ the~ paper:$ After recalling the basic definitions and properties of Oliver's $p$-group in Section 2,
and we lists some properties of Sylow subgroups
of unitary groups and calculations of $\mathfrak{X}(S)$ in Section 3.

\section{\bf Oliver's $p$-group and some lemmas}

In this section we collect some contents that will be need later, we refer to \cite{GHL, GHM,O}.
First, we recall the definition of Oliver's $p$-group as follows.
\begin{definition}\cite[Definition 3.1]{O},\cite[Definition]{GHL} Let $S$ be a finite $p$-group and
$K\unlhd S$ a normal subgroup. There exists a sequence
$$1=Q_0\leq Q_1\leq\cdots\leq Q_n=\mathfrak{X}(S)$$
such that $Q_i\unlhd S$, and such that
$$[\Omega_1(C_S(Q_{i-1})), Q_i; p-1]=1\quad\quad \quad\quad \quad\quad (\ast)$$
holds for each $1\leq i\leq n$. This series is called as {\bf $Q$-series}.
The unique largest normal subgroup of $S$ which admits such a sequence
is called $\mathfrak{X}(S)$, {\bf the Oliver subgroup} of $S$.
\end{definition}

\begin{remark}\cite[p.334]{O} If $K, L\unlhd S$ and there exist two sequences
$$1=Q_0\leq Q_1\leq\cdots\leq Q_n=K;~~~1=R_0\leq R_1\leq\cdots\leq R_m=L$$
such that $Q_i\unlhd S, R_j\unlhd S$, and such that
$$[\Omega_1(C_S(Q_{i-1})), Q_i; p-1]=1~~~ \mathrm{and} ~~~[\Omega_1(C_S(R_{j-1})), R_j; p-1]=1$$
holds for each $1\leq i\leq n$ and $1\leq j\leq m$ respectively.
Then we have a sequence
$$1=Q_0\leq Q_1\leq\cdots\leq Q_n=K\leq Q_nR_0\leq Q_nR_1\leq\cdots\leq Q_nR_m=KL$$
satisfies the condition $(\ast)$.
\end{remark}

The following lemmas are important to prove the main results.
\begin{lemma}\cite[Lemma 3.2]{O} Let $S$ be a finite $p$-group, then $C_{S}(\mathfrak{X}(S))=Z(\mathfrak{X}(S))$.
\end{lemma}

\begin{lemma}\cite[Lemma 3.3]{O} Let $S$ be a finite $p$-group. Let $Q\unlhd S$ be any normal subgroup
such that
$$[\Omega_1(Z(\mathfrak{X}(S))), Q; p-1]=1.$$
Then $Q\leq \mathfrak{X}(S)$.
\end{lemma}

\begin{theorem}\cite[Theorem 1.7]{GHM2}Let $p$ be a prime and G a finite $p$-group. If $G$ is generated by its abelian normal subgroups, then
Then Conjecture 1.4 holds for $G$.
\end{theorem}

\begin{theorem}\cite[Lemma 5.4]{GHM2} Suppose that $p$ is an odd prime and that $P$ is a $p$-group such that $J(P)$ is elementary abelian.
Then $J(P\wr C_p)$ is elementary abelian too. In particular, if $P\neq 1$ then $J(P\wr C_p)$ is the copy of $J(P)^p$ in the base
subgroup $P^p\leq  P \wr C_p$.
\end{theorem}

\section{\bf Unitary groups, defining characteristic}

Following the terminology of Reid(\cite{R}), we refer to the bottom left to top right diagonal of
a square matrix as the skew-diagonal; reflecting $B$ in its skew-diagonal yields the flip-transpose $B^F$;  $B$ is $persymmetric$
if  $B^F=B$; and $B$ is $skew$-$persymmetric$
if  $B^F=-B$.  For fixed $m$, we write $Q$ for
$Q_m$, meaning that $Q$ is the $(m\times m)$-matrix with $1$ in all entries on the skew-diagonal, and
$0$ everywhere else.

\begin{proposition}
Letters $B, C,\ldots, $ will refer to $(m\times m)$-matrices. The following hold:

(1) $Q^2=\mathbbm{1}$
where $\mathbbm{1}$ signify the
identity element of $\mathrm{GL}_m(k)$.

(2) $B$ is persymmetric if only if $QB$ is symmetric; $B$ is persymmetric if only if $BQ$ is symmetric.
$B$ is skew-persymmetric if only if $QB$ is skew-symmetric; $B$ is skew-persymmetric if only if $BQ$ is skew-symmetric.

(3) $QB^TQ=B^F$.

(4) $(BC)^F=C^FB^F$.

(5) $(B^F)^{-1}=(B^{-1})^F$ for invertible $B$.
\end{proposition}

{\bf Notation: we sometime write $Q$ for
$Q_m$, meaning that $Q$ is the $(m\times m)$-matrix with $1$ in all entries on the skew-diagonal, and
$0$ everywhere else.}

~

Recall that the underlying field is $\mathbb{F}_{q^2}$ which has an automorphism $x\mapsto \overline{x}=x^q$ of order 2. For a matrix
$A=(a_{ij})$, $\overline{A}$ means $\overline{A}=(\overline{a_{ij}})$. We call $A$ is $conjugate$-$skew$-$persymmetric$
if  $\overline{A}^F=-A$.
The definition of general unitary group $U_{n}(\mathbb{F}_q)$ can be found in \cite{Car,Wang, Wi}.
We can recall the definition of unitary group  as follows.
$$\mathrm{U}_{n}(\mathbb{F}_q):=\{A\in \mathrm{GL}_n(\mathbb{F}_{q^2})|\overline{A}^{\mathrm{tr}}I_{n}A=I_{n}\}$$
where $I_n$ is identity matrix.

Now, set
$$\Gamma_n(Q_n, \mathbb{F}_{q^2}):=\{A\in \mathrm{GL}_n(\mathbb{F}_{q^2})|\overline{A}^{\mathrm{tr}}Q_{n}A=Q_{n}\}.$$
In fact, we can find invertible matrix $B$ such that $\overline{B}^{tr}Q_{n}B=I_n$. So, we have
$$\Gamma_n(Q_n, \mathbb{F}_{q^2})\cong \mathrm{U}_{n}(\mathbb{F}_q)$$
by the map $A\longmapsto B^{-1}AB.$

For the convenience of calculation, we will set
$$\mathrm{U}_{n}(\mathbb{F}_q)=\{A\in \mathrm{GL}_n(\mathbb{F}_{q^2})|\overline{A}^{\mathrm{tr}}Q_{n}A=Q_{n}\}$$
in the following sections

\subsection{For $n=2m$}

If $q$ is a power of $p$, we can calculate the Sylow $p$-subgroup $S$ of $\mathrm{U}_{2m}(\mathbb{F}_{q})$ as follows.

\begin{theorem}
Let $G=\mathrm{U}_{2m}(\mathbb{F}_{q})$ and  $q$ be a power of $p$. Let $\mathcal{S}$ be the Sylow $p$-subgroup  of $G$ as subgroup $G \cap \mathrm{SL}_{2m}(\mathbb{F}_{q^2})$ of the group $\mathrm{SL}_{2m}(\mathbb{F}_{q^2})$ of lower
unitriangular $2m\times 2m$-matrices over $\mathbb{F}_{q^2}$. Then the element of $\mathcal{S}$ has following form:
$$X_{D, P}:=\left(%
\begin{array}{cc}
(\overline{D}^F)^{-1} &  0 \\
DP &  D  \\
\end{array}%
\right)$$
where $D$ is a lower
unitriangular $m\times m$-matrix, and $P$ is a conjugate-skew-persymmetric $m\times m$-matrix.
\end{theorem}

\begin{proof}Let $X=\left(%
\begin{array}{cc}
B &  0 \\
C &  D  \\
\end{array}%
\right)\in \mathrm{GL}_{2m}(\mathbb{F}_{q^2})$. We have $X\in \mathrm{U}_{2m}(\mathbb{F}_{q})$ if and only if
$$\left(%
\begin{array}{cc}
\overline{B}^T &  \overline{C}^T \\
0 &  \overline{D}^T  \\
\end{array}%
\right)\left(%
\begin{array}{cc}
0 &  Q \\
Q &  0  \\
\end{array}%
\right)\left(%
\begin{array}{cc}
B &  0 \\
C &  D  \\
\end{array}%
\right)=\left(%
\begin{array}{cc}
0 &  Q \\
Q &  0  \\
\end{array}%
\right).$$
That means
\begin{eqnarray*}
&~& \overline{B}^TQC=-\overline{C}^TQB~~~~~~~~~~~~~~~~~~~(\ast)\\
&~& \overline{B}^TQD=Q.~~~~~~~~~~~~~~~~~~~~~~~~~(\ast\ast)
\end{eqnarray*}
By $(\ast\ast)$ and Proposition 3.1(3), we have $B=(\overline{D}^F)^{-1}$ because $D$ and $B$ are invertible.
Set $P:=Q\overline{B}^TQC$, we have $P=D^{-1}C$ because $B=(\overline{D}^F)^{-1}$ and Proposition 3.1(3). By $(\ast)$ and Proposition 3.1(2),
we have $QP^TQ=-\overline{P}$. So, $P$ is conjugate-skew-persymmetric. Set $$X_{D, P}:=\left(%
\begin{array}{cc}
(\overline{D}^F)^{-1} &  0 \\
DP &  D  \\
\end{array}%
\right)$$
We can see that there are  $q^{m(m-1)}$
such $D$ and $(q^2)^{\frac{m(m-1)}{2}+\frac{m}{2}}$ such $P$(Here, the reason is that $QP^TQ=-\overline{P}$), we see
that there are $q^{m(2m-1)}=q^{\frac{n(n-1)}{2}}$
lower-triangular unipotent matrices in  $\mathrm{U}_{2m}(\mathbb{F}_q)$; and since these clearly
form a subgroup. So, we have $\mathcal{S}$ is a Sylow $p$-subgroup of $\mathrm{U}_{2m}(\mathbb{F}_q)$ because the order of $\mathrm{U}_{2m}(\mathbb{F}_q)$
is $q^{\frac{n(n-1)}{2}}\prod_{i=1}^n(q^i-(-1)^i)$(see\cite[p.66]{Wi}).
\end{proof}

We will prove that Sylow subgroup is a semidirect product.

\begin{theorem}
Let $G=\mathrm{U}_{2m}(\mathbb{F}_{q})$ and  $q$ be a power of $p$. Let $\mathcal{S}$ be the Sylow $p$-subgroup of $G$ defined as above.
Set $\mathcal{A}=:\{X_{\mathbbm{1}, P}|P\in M_m(\mathbb{F}_{q^2})~\mathrm{and}~P\mathrm{~is}$-$\mathrm{conjugate}$-$\mathrm{skew}$-$\mathrm{persymmetric}\}$ and
$\mathcal{D}:=\{X_{D, 0}|D\in M_m(\mathbb{F}_{q^2})~\mathrm{and}~D ~\mathrm{is~lower~unitriangular}\}$.
Then $\mathcal{S}\cong \mathcal{A}\rtimes \mathcal{D}$.
\end{theorem}

\begin{proof}First, if $P$ is conjugate-skew-persymmetric, we can see that $X_{\mathbbm{1}, P}\in \mathrm{U}_{2m}(\mathbb{F}_q)$, that means $X_{\mathbbm{1}, P}\in \mathcal{S}$.
Let $P'$ be another conjugate-skew-persymmetric, we have
$$X_{\mathbbm{1}, P}X_{\mathbbm{1}, P'}=X_{\mathbbm{1}, P+P'}=X_{\mathbbm{1}, P'}X_{\mathbbm{1}, P}.$$
So, $\mathcal{A}$ is abelian.

Let $D, D'$ be $D$ is a lower
unitriangular $m\times m$-matrices,
we have
\begin{eqnarray*}
X_{D, P}X_{D',P'}
&=& \left(%
\begin{array}{cc}
(\overline{D}^F)^{-1} &  0 \\
DP &  D  \\
\end{array}%
\right)\left(%
\begin{array}{cc}
(\overline{D}'^F)^{-1} &  0 \\
D'P' &  D'  \\
\end{array}%
\right)\\
&= & \left(%
\begin{array}{cc}
((\overline{DD'})^F)^{-1} &  0 \\
DP(\overline{D'}^F)^{-1}+DD'P' &  DD'  \\
\end{array}%
\right)\\
&=& X_{DD', D'^{-1}P(\overline{D'}^F)^{-1}+P'}.
\end{eqnarray*}
Hence, we have
$$X_{D, P}^{-1}=X_{D^{-1}, -DP\overline{D}^F}.$$
We can compute the commutator as follows:
\begin{eqnarray*}
[X_{\mathbbm{1}, P}, X_{D, P'}]
&=& X_{\mathbbm{1}, P}^{-1} X_{D, P'}^{-1} X_{\mathbbm{1}, P} X_{D, P'}\\
&=& X_{\mathbbm{1}, D^{-1}P(\overline{D}^{-1})^F-P}. ~~~~~~~~~~~~~~~~~~(\ast)
\end{eqnarray*}
Hence, we have $\mathcal{A}$ is a normal abelian subgroup of $\mathcal{S}$.

Secondly, we can see that
$$X_{D, 0}X_{D',0}=X_{DD', 0}.$$
That means $\mathcal{D}$ is a subgroup of $\mathcal{S}$.
And we can see that $$X_{D, P}=\left(%
\begin{array}{cc}
(\overline{D}^F)^{-1} &  0 \\
DP &  D  \\
\end{array}%
\right)=\left(%
\begin{array}{cc}
(\overline{D}^F)^{-1} &  0 \\
0&  D  \\
\end{array}%
\right)\left(%
\begin{array}{cc}
\mathbbm{1} &  0 \\
P &  \mathbbm{1}  \\
\end{array}%
\right)=X_{D, 0}X_{\mathbbm{1}, P}.$$
So, $\mathcal{S}=\mathcal{D}\rtimes\mathcal{A}$ because
$\mathcal{D}\cap \mathcal{A}=1$.
\end{proof}

We will prove that subgroup $\mathcal{A}$ is self-centralising.

\begin{theorem}Let $m$ be an integer with $m\geq 2$.
Let $\mathcal{S}$ and $\mathcal{A}$ be defined as above.
Then $C_{\mathcal{S}}(\mathcal{A})\leq \mathcal{A}$.
\end{theorem}

\begin{proof}
 Let $X_{D, P'}\in C_{\mathcal{S}}(X_{\mathbbm{1}, P})$, we have
$$[X_{\mathbbm{1}, P}, X_{D, P'}]=1.$$
By Theorem 3.3$(\ast)$, we have $D^{-1}P(\overline{D}^{-1})^F=P.$ That means
$DP\overline{D}^F=P.$ Since $D$ is a lower
unitriangular $m\times m$-matrix, we can set $D:=\mathbbm{1}+U$ for some lower triangular nilpotent matrix $U$.
So, $P=DP\overline{D}^F=(\mathbbm{1}+U)P(\mathbbm{1}+\overline{U}^F)$, that means $UP+P\overline{U}^F+UP\overline{U}^F=0$.
Hence, we have
$$X_{D, P'}\in C_{\mathcal{S}}(X_{\mathbbm{1}, P})\Longleftrightarrow UP+P\overline{U}^F+UP\overline{U}^F=0 ~~~~~~~~~~~~~~~~~(\ast)$$
where $D=\mathbbm{1}+U$.

Suppose that $X_{D, P'}\in C_{\mathcal{S}}(\mathcal{A})- \mathcal{A}$,
that means $D\neq 1$. Then $U\neq 0$.

To apply a similar method as in Theorem 3.4, we define $P\in M_m(\mathbb{F}_q)$ to be the matrix given by
$$P_{ab}=\delta_{as}\delta_{b1}-\delta_{am}\delta_{bs'}.$$
That means
 $$P:=\left(%
\begin{array}{ccc}
   &   &  \\
1  &   &  \\
&  -1 &  \\
\end{array}%
\right)\leftarrow\mathrm{(~in~ s^{th}~ row)}$$
$$\uparrow~~~~~~~~~~~$$
$$\mathrm{(~in~ s'^{th}~ column)}$$
So, $P$ is a skew-persymmetric $m\times m$ matix. That means $X_{1, P}\in\mathcal{A}$.
We will prove that $UP+P\overline{U}^F+UP\overline{U}^F\neq 0$ to get the contradiction.

We can see that
$$(UP)_{ab}=\sum_{t=1}^m U_{at}P_{tb}=U_{as}\delta_{b1}-U_{am}\delta_{bs'}=U_{as}\delta_{b1};  ~~~~~~\mathrm{For}~U_{am}=0$$
$$(P\overline{U}^F)_{ab}=-(\overline{U}P)_{b'a'}=-\overline{U}_{b's}\delta_{a'1}=-\overline{U}_{b's}\delta_{am}=\begin{cases}
-\overline{U}_{b's}& \text{if}~a=m,\\
0 & \text{otherwise;}
\end{cases}$$
$$(UP\overline{U}^F)_{ab}=\sum_{c=1}^m U_{ac}(P\overline{U}^F)_{cb}=U_{am}(P\overline{U}^F)_{mb}=0. ~~~~~~\mathrm{For}~U_{am}=0$$
Hence,
\begin{eqnarray*}
(UP+P\overline{U}^F+UP\overline{U}^F)_{r1}
&=& U_{rs}- U_{ms}\delta_{rm}+0\\
&=& \begin{cases}
0 & \text{if}~r=m,\\
U_{rs} & \text{if}~r\neq m
\end{cases}
\end{eqnarray*}
Since $X_{D, P'}\in C_{\mathcal{S}}(\mathcal{A})$, we have $U_{rs}=0$ if $r\neq m$.

Also, if $m\geq 2$, we have
\begin{eqnarray*}
(UP+P\overline{U}^F+UP\overline{U}^F)_{r2}
&=& U_{rs}\delta_{21}- U_{ms}\delta_{rm}+0\\
&=& \begin{cases}
-U_{ms} & \text{if}~r=m,\\
0 & \text{if}~r\neq m
\end{cases}
\end{eqnarray*}
Since $X_{D, P'}\in C_{\mathcal{S}}(\mathcal{A})$, we have $U_{ms}=0$.

Therefore, we have $U$ is zero matrix, that is a contradiction. So, $C_{\mathcal{S}}(\mathcal{A})\leq \mathcal{A}$.
\end{proof}

Now, we calculate $\mathfrak{X}(\mathcal{S})$ as follows.

\begin{theorem} Let $p$ be a prime with $p\geq 5$.
Let $\mathcal{S}$ and $\mathcal{A}$ be defined as above.
Then $\mathfrak{X}(\mathcal{S})=\mathcal{S}$.
\end{theorem}

\begin{proof} From \cite{Weir} and \cite{GHM2}, we can collect some abelian normal subgroups $\mathcal{N}_{ij}$ of the group of lower unitriangular matrices of $\mathrm{SL}_m(\mathbb{F}_{q^2})$ for $1\leq j < i \leq m$, where
$$\mathcal{N}_{ij}= \{M\in \mathrm{SL}_m(\mathbb{F}_{q^2})| M_{aa}=1; M_{ab}=0 ~\mathrm{for}~ a\neq b~ \mathrm{whenever}~ a < i ~\mathrm{or}~ b > j  \}$$
where  $M_{ab}\in \mathbb{F}_q$ denotes the $(a, b)$-entry of $M$
(the entry in the $a$th row and $b$th column).
Set $\widetilde{\mathcal{N}}_{ij}:=\{X_{D, P}|D\in\mathcal{N}_{ij}, \mathrm{and}~P \mathrm{~is}~$$\mathrm{conjugate}$-$\mathrm{skew}$-$\mathrm{persymmetric}\}$. We can see that
$\langle \widetilde{\mathcal{N}}_{ij}|1\leq j < i \leq m\rangle=\mathcal{S}$.
It therefore suffices to show that $1\leq \mathcal{A}\leq\mathcal{N}_{ij}$ is always a $Q$-series.

We will prove that
$$[[[X_{\mathbbm{1}, P}, X_{D , \hat{P}}], X_{D', \hat{P}'}],  X_{D^{''}, \hat{P}^{''}}]=1$$
for any $D, D', D^{''}\in \mathcal{N}_{ij}$ and any conjugate-skew-persymmetric matrices $P, \hat{P}, \hat{P}', \hat{P}^{''}$.

To calculate, we can set
$$D^{-1}=\mathbbm{1}+U,~ D^{'-1}=\mathbbm{1}+U',~ D^{''-1}=\mathbbm{1}+U^{''}$$
where $U, U', U^{''}$ are the elements of $M_m(\mathbb{F}_{q^2})$ and their entries satisfies
 $U_{ab} = 0, U'_{ab}=0, U^{''}_{ab}=0$ whenever
$a < i $ or $b > j$. So we can write $U, U', U^{''}$ as following block matrix:
$$U=\left(%
\begin{array}{cc}
0 & 0   \\
T & 0 \\
\end{array}%
\right),~
U'=\left(%
\begin{array}{cc}
0 & 0   \\
T' & 0 \\
\end{array}%
\right)~,
U^{''}=\left(%
\begin{array}{cc}
0 & 0   \\
T^{''} & 0 \\
\end{array}%
\right)$$
where $T$, $T'$ and $T^{''}$ are $(m-j)\times j$ matrices because $i>j$.

By Theorem 3.3$(\ast)$, we have
$$[X_{\mathbbm{1}, P}, X_{D , \hat{P}}]= X_{\mathbbm{1}, P'} ~~\mathrm{for}~~ P'=UP+P\overline{U}^F+UP\overline{U}^F.$$
To compute $P'$, $P$ can be viewed as following block matrix:
$$P=\left(%
\begin{array}{cc}
P_1 & P_2 \\
P_3 & P_4 \\
\end{array}%
\right)$$
where $P_1$ is a $j\times (m-j)$ matrix, $P_2$ is a $j\times j$ matrix, $P_3$ is a $(m-j)\times (m-j)$ matrix, and  $P_4$ is a $(m-j)\times j$ matrix.
So, we can see that
$$UP=\left(%
\begin{array}{cc}
0 & 0   \\
T & 0 \\
\end{array}%
\right)\left(%
\begin{array}{cc}
P_1 & P_2 \\
P_3 & P_4 \\
\end{array}%
\right)
=\left(%
\begin{array}{cc}
0 & 0   \\
TP_1 & TP_2 \\
\end{array}%
\right);
$$
$$P\overline{U}^F=\left(%
\begin{array}{cc}
P_1 & P_2 \\
P_3 & P_4 \\
\end{array}%
\right)\left(%
\begin{array}{cc}
0 & 0   \\
\overline{T}^F & 0 \\
\end{array}%
\right)
=\left(%
\begin{array}{cc}
P_2\overline{T}^F & 0   \\
P_4\overline{T}^F & 0 \\
\end{array}%
\right);
$$
$$UP\overline{U}^F=\left(%
\begin{array}{cc}
0 & 0   \\
T & 0 \\
\end{array}%
\right)\left(%
\begin{array}{cc}
P_1 & P_2 \\
P_3 & P_4 \\
\end{array}%
\right)\left(%
\begin{array}{cc}
0 & 0   \\
\overline{T}^F & 0 \\
\end{array}%
\right)
=\left(%
\begin{array}{cc}
0 & 0   \\
TP_2\overline{T}^F & 0 \\
\end{array}%
\right).
$$
So, we have
$$P'=\left(%
\begin{array}{cc}
P_2\overline{T}^F & 0   \\
TP_1+P_4\overline{T}^F+TP_2\overline{T}^F & TP_2 \\
\end{array}%
\right). ~~~~~~~~~~~~~~(\ast)$$

Also by Theorem 3.3$(\ast)$, we have
$$[X_{\mathbbm{1}, P'}, X_{D' , \hat{P}'}]= X_{\mathbbm{1}, P''} ~~\mathrm{for}~~ P{''}:=U'P'+P'\overline{U'}^F+U'P'\overline{U'}^F.$$
$P'$ can be viewed as following block matrix:
$$P'=\left(%
\begin{array}{cc}
P'_1 & P'_2 \\
P'_3 & P'_4 \\
\end{array}%
\right):=\left(%
\begin{array}{cc}
P_2\overline{T}^F & 0   \\
TP_1+P_4\overline{T}^F+TP_2\overline{T}^F & TP_2 \\
\end{array}%
\right).$$
By $(\ast)$, we have
$$P^{''}=\left(%
\begin{array}{cc}
0 & 0   \\
T'P_2\overline{T}^F+TP_2\overline{T'}^F & 0 \\
\end{array}%
\right).~~~~~~~~~~~~~(\ast\ast)$$

Also by Theorem 3.3$(\ast)$, we have
$$[X_{\mathbbm{1}, P^{''}}, X_{D^{''} , \hat{P}^{''}}]= X_{\mathbbm{1}, P^{'''}} ~~\mathrm{for}~~ P{'''}:=U^{''}P^{''}+P^{''}\overline{U^{''}}^{F}+U^{''}P^{''}\overline{U^{''}}^{F}.$$
By $(\ast\ast)$,  we have
$$P^{'''}=\left(%
\begin{array}{cc}
0 & 0   \\
T^{''}P^{'}_2\overline{T^{'}}^{F}+T^{'}P^{'}_2\overline{T^{''}}^{F} & 0 \\
\end{array}%
\right)=0.~~~~~~~~~~~~\mathrm{For}~P'_2=0$$
That means
$$[[[X_{\mathbbm{1}, P}, X_{D , \hat{P}}], X_{D', \hat{P}'}],  X_{D^{''}, \hat{P}^{''}}]=X_{\mathbbm{1}, P^{'''}}=X_{\mathbbm{1}, 0}=1. $$
If $m=1$, it is easy to see that $\mathfrak{X}(\mathcal{S})=\mathcal{S}$ because $p\geq 5$.
So, if $m\geq 2$, that means $C_{\mathcal{S}}(\mathcal{A})\leq \mathcal{A}$ by Theorem 3.4.
Hence,
$$[\mathcal{A},\widetilde{\mathcal{N}}_{ij}; 3]=1.$$
That means $1\leq \mathcal{A}\leq \widetilde{\mathcal{N}}_{ij}$ is always a $Q$-series because $p\geq 5$.
Hence, $\mathfrak{X}(\mathcal{S})=\mathcal{S}$.
\end{proof}

\subsection{\bf For $n=2m+1$}

$Notation.$ Let $\alpha$ be any $1\times m$ vector, if $P$ is called $\alpha$-$conjugate$-$skew$-$persymmetric$ if
$$P+\overline{P}^F=-Q\overline{\alpha}^T\alpha$$
where $P$ is an $m\times m$ matrix and $Q$ defined as above. We can see that $0$-conjugate-skew-persymmetric
means conjugate-skew-persymmetric.

If $q$ is a power of $p$, we can calculate the Sylow $p$-subgroup $S$ of $\mathrm{U}_{2m+1}(\mathbb{F}_{q})$ as follows.
\begin{theorem}
Let $G=\mathrm{U}_{2m+1}(\mathbb{F}_{q})$ and  $q$ be a power of $p$. Let $\mathcal{S}$ be the Sylow $p$-subgroup  of $G$ as subgroup $G \cap \mathrm{SL}_{2m+1}(\mathbb{F}_{q^2})$ of the group $\mathrm{SL}_{2m+1}(\mathbb{F}_{q^2})$ of lower
unitriangular $(2m+1)\times (2m+1)$-matrices over $\mathbb{F}_{q^2}$. Then  the element of $\mathcal{S}$ has following form:
$$X_{D, P, \alpha}:=\left(%
\begin{array}{ccc}
(\overline{D}^F)^{-1} &  0 & 0\\
\alpha   & 1  & 0 \\
DP & -DQ\overline{\alpha}^T & D  \\
\end{array}%
\right)$$
where $D$ is a lower
unitriangular $m\times m$-matrix, $\alpha$ is a $1\times m$ vector and $P$ is a $\alpha$-conjugate-skew-persymmetric $m\times m$-matrix.
\end{theorem}

\begin{proof}Let $X=\left(%
\begin{array}{ccc}
B &  0 & 0\\
\alpha  & 1  & 0 \\
 C & \beta & D  \\
\end{array}%
\right)\in M_{2m+1}(\mathbb{F}_{q^2})$. We have $X\in \mathrm{U}_{2m+1}(\mathbb{F}_{q})$ if and only if
$$\left(%
\begin{array}{ccc}
\overline{B}^T &  \overline{\alpha}^T & \overline{C}^T\\
0 & 1  & \overline{\beta}^T \\
0 & 0 & \overline{D}^T  \\
\end{array}%
\right)\left(%
\begin{array}{ccc}
0   & 0&  Q \\
0  &  1  &  0\\
Q & 0 & 0  \\
\end{array}%
\right)\left(%
\begin{array}{ccc}
B &  0 & 0\\
\alpha  & 1  & 0 \\
 C & \beta & D  \\
\end{array}%
\right)=\left(%
\begin{array}{ccc}
0   & 0&  Q \\
0  &  1  &  0\\
Q & 0 & 0  \\
\end{array}%
\right).$$
That means
\begin{eqnarray*}
&~& \overline{B}^TQC+\overline{C}^TQB=-\overline{\alpha}^T\alpha~~~~~~~~~~~~~~~~~~~(\ast)\\
&~& \overline{B}^TQD=Q~~~~~~~~~~~~~~~~~~~~~~~~~~~~~~~~~~~~~(\ast\ast)\\
&~& \beta=-Q(\overline{B}^{-1})^T\overline{\alpha}^T.~~~~~~~~~~~~~~~~~~~~~~~~~~~(\ast\ast\ast)
\end{eqnarray*}
By $(\ast\ast)$ and Proposition 3.1(3), we have $B=(D^F)^{-1}$ because $D$ and $B$ are invertible. By $(\ast\ast\ast)$, we have $\beta=-DQ\overline{\alpha}^T$.
Set $P:=Q\overline{B}^TQC$, we have $P=D^{-1}C$ because $B=(\overline{D}^F)^{-1}$ and Proposition 3.1(3).
By $(\ast)$ and Proposition 3.1(2), we have
$$P+\overline{P}^F=-Q\overline{\alpha}^T\alpha.$$
So, we have $P$ is $\alpha$-conjugate-skew-persymmetric.
Now, we can see that
$$X=\left(%
\begin{array}{ccc}
(\overline{D}^F)^{-1} &  0 & 0\\
\alpha  & 1  & 0 \\
 DP & -DQ\overline{\alpha}^T & D  \\
\end{array}%
\right)= \left(%
\begin{array}{ccc}
(\overline{D}^F)^{-1} &  0 & 0\\
0   & 1  & 0 \\
 0 & 0 & D  \\
\end{array}%
\right)\left(%
\begin{array}{ccc}
\mathbbm{1} &  0 & 0\\
\alpha   & 1  & 0 \\
P & -Q\overline{\alpha}^T & \mathbbm{1} \\
\end{array}%
\right)  $$
That means
$X=X_{D, P, \alpha}=X_{D, 0, 0}X_{\mathbbm{1}, P, \alpha}$.

Set $\mathcal{A}=:\{X_{\mathbbm{1}, P,\alpha}|\alpha \mathrm{~is ~a~}1\times m~ \mathrm{vector}, P \mathrm{~is}~\alpha$-$\mathrm{conjugate}$-$\mathrm{skew}$-$\mathrm{persymmetric}\}$,
 $\mathcal{A}_0=:\{X_{\mathbbm{1}, P, 0}| P \mathrm{~is}$ $\mathrm{conjugate}$-$\mathrm{skew}$-$\mathrm{persymmetric}\}$, we can see that $\mathcal{A}_0$ and $\mathcal{A}$
 are subgroups of $\mathcal{S}$.

 Let $X_{\mathbbm{1}, P, \alpha}$ and $X_{\mathbbm{1},P',\alpha'}$ be  elements of $\mathcal{A}$,
we have
\begin{eqnarray*}
[X_{\mathbbm{1}, P, \alpha}, X_{\mathbbm{1}, P', \alpha'}]
&=& X_{\mathbbm{1}, P, \alpha}^{-1} X_{\mathbbm{1}, P', \alpha'}^{-1} X_{\mathbbm{1}, P, \alpha} X_{\mathbbm{1}, P', \alpha'}\\
&=& X_{\mathbbm{1}, -P-Q\overline{\alpha}^T\alpha, -\alpha} X_{\mathbbm{1}, -P'-Q\overline{\alpha'}^T\alpha', -\alpha'} X_{\mathbbm{1}, P, \alpha} X_{\mathbbm{1}, P', \alpha'}\\
&=& X_{\mathbbm{1}, -P-P'-Q\overline{\alpha}^T\alpha-Q\overline{\alpha'}^T\alpha'-Q\overline{\alpha}^T\alpha', -\alpha-\alpha'}  X_{\mathbbm{1}, P+P'-Q\overline{\alpha}^T\alpha', \alpha+\alpha'}\\
&=& X_{\mathbbm{1}, Q\overline{\alpha'}^T\alpha-Q\overline{\alpha}^T\alpha', 0}. ~~~~~~~~~~~~~~~~~~(\star)
\end{eqnarray*}
That means $[\mathcal{A}_0,\mathcal{A}]=1$, $[\mathcal{A}, \mathcal{A}]\leq \mathcal{A}_0$ and $\mathcal{A}_0\unlhd \mathcal{A}$.

We can see that there are  $(q^2)^{\frac{m(m-1)}{2}}$
such $D$,  $(q^2)^{\frac{m(m-1)}{2}+\frac{m}{2}}$ such $\mathcal{A}_0$ and $(q^2)^m$ such $\mathcal{A}/\mathcal{A}_0$,  we see
that there are $q^{m(m-1)+m^2+2m}=q^{m(2m+1)}=q^{\frac{n(n-1)}{2}}$
lower-triangular unipotent matrices in  $\mathrm{U}_{2m+1}(\mathbb{F}_q)$; and since these clearly
form a subgroup. So, we have $\mathcal{S}$ is a Sylow $p$-subgroup of $\mathrm{U}_{2m+1}(\mathbb{F}_q)$.
\end{proof}

We will prove some properties of Sylow subgroup as follows.
\begin{theorem}
Let $G=\mathrm{U}_{2m+1}(\mathbb{F}_{q})$ and  $q$ be a power of $p$. Let $\mathcal{S}$ be the Sylow $p$-subgroup of $G$ defined as above. Let
$\mathcal{A}$ and $\mathcal{A}_0$ be defined as above. Set
$\mathcal{D}:=\{X_{D, 0, 0}|$ $D ~\mathrm{is~lower~unitriangular}\}$. Then
we have

(1) $\mathcal{A}_0$ is an abelian normal subgroup of $\mathcal{S}$;

(2) Let $ X_{D, P', \alpha}\in \mathcal{S}$, then $[X_{\mathbbm{1}, P, 0}, X_{D, P', \alpha}]=[X_{\mathbbm{1}, P, 0}, X_{D,  0, 0}]$ for each $P$ 0-conjugate-skew-persymmetric, $P'$ $\alpha$-conjugate-skew-persymmetric;

(3) $\mathcal{S}\cong \mathcal{A}\rtimes \mathcal{D}$.
\end{theorem}

\begin{proof}First, if $P$ is 0-conjugate-skew-persymmetric, we can see that $X_{\mathbbm{1}, P, 0}\in \mathrm{U}_{2m+1}(\mathbb{F}_{q})$, that means $X_{\mathbbm{1}, P, 0}\in \mathcal{S}$.
Let $P'$ be another 0-conjugate-skew-persymmetric, we have
$$X_{\mathbbm{1}, P, 0}X_{\mathbbm{1}, P', 0}=X_{\mathbbm{1}, P+P', 0}=X_{\mathbbm{1}, P', 0}X_{\mathbbm{1}, P, 0}.$$
So, $\mathcal{A}_0$ is abelian.

By Theorem 3.6, we have $X_{D, P', \alpha}=X_{D, 0, 0}X_{\mathbbm{1}, P', \alpha}$.
 We can calculate the commutator as follows:
\begin{eqnarray*}
&~&[X_{\mathbbm{1}, P, 0}, X_{D, 0, 0}]
= X_{\mathbbm{1}, P, 0}^{-1} X_{D, 0, 0}^{-1} X_{\mathbbm{1}, P, 0} X_{D, 0, 0}\\
&=& \left(%
\begin{array}{ccc}
\mathbbm{1} &  0 & 0\\
0  & 1  & 0 \\
-P & 0 & \mathbbm{1} \\
\end{array}%
\right)\left(%
\begin{array}{ccc}
\overline{D}^F &  0 & 0\\
0  & 1  & 0 \\
0 & 0 & D^{-1} \\
\end{array}%
\right)    \left(%
\begin{array}{ccc}
\mathbbm{1} &  0 & 0\\
0  & 1  & 0 \\
P & 0 & \mathbbm{1} \\
\end{array}%
\right)\left(%
\begin{array}{ccc}
(\overline{D}^F)^{-1} &  0 & 0\\
0  & 1  & 0 \\
0 & 0 & D \\
\end{array}%
\right)               \\
&=&\left(%
\begin{array}{ccc}
\overline{D}^F &  0 & 0\\
0   & 1  & 0 \\
-P\overline{D}^F & 0 & D^{-1} \\
\end{array}%
\right)\left(%
\begin{array}{ccc}
(\overline{D}^F)^{-1} &  0 & 0\\
0   & 1  & 0 \\
P(\overline{D}^F)^{-1} & 0 & D \\
\end{array}%
\right)\\
&=& \left(%
\begin{array}{ccc}
\mathbbm{1} &  0 & 0\\
0   & 1  & 0 \\
-P+ D^{-1}P(\overline{D}^F)^{-1} & 0 & \mathbbm{1} \\
\end{array}%
\right)\\
&=& X_{\mathbbm{1}, -P+ D^{-1}P(\overline{D}^F)^{-1}, 0}.~~~~~~~~~~~~~~~~~~~~~~~~~~~~~~~~~~~(\ast)
\end{eqnarray*}

Since $[\mathcal{A}_0,\mathcal{A}]=1$, we can see that
\begin{eqnarray*}
[X_{\mathbbm{1}, P, 0}, X_{D, P', \alpha}]
&=& [X_{\mathbbm{1}, P, 0}, X_{D, 0, 0}X_{\mathbbm{1}, P', \alpha}]\\
&=&[X_{\mathbbm{1}, P, 0}, X_{\mathbbm{1}, P', \alpha}][X_{\mathbbm{1}, P, 0}, X_{D, 0, 0}]^{X_{\mathbbm{1}, P', \alpha}}\\
&=&[X_{\mathbbm{1}, P, 0}, X_{D, 0, 0}]^{X_{\mathbbm{1}, P', \alpha}} \\
&=& [X_{\mathbbm{1}, P, 0}, X_{D, 0, 0}].~~~~~~~~~~~~~~~~~~~~~~~\mathrm{For}~ (\ast)
\end{eqnarray*}
Hence, (2) holds.

We can calculate the commutator as follows:
\begin{eqnarray*}
&~&[X_{\mathbbm{1}, P, \alpha}, X_{D, 0, 0}]= X_{\mathbbm{1}, P, \alpha}^{-1} X_{D, 0, 0}^{-1} X_{\mathbbm{1}, P, \alpha} X_{D, 0, 0}\\
&=& X_{\mathbbm{1}, -P-Q\overline{\alpha}^T\alpha, -\alpha} X_{D^{-1}, 0, 0} X_{\mathbbm{1}, P, \alpha} X_{D, 0, 0}\\
&=&\left(%
\begin{array}{ccc}
\overline{D}^F &  0 & 0\\
-\alpha \overline{D}^F   & 1  & 0 \\
-P\overline{D}^F -Q\overline{\alpha}^T\alpha D^F & Q\overline{\alpha}^T & D^{-1} \\
\end{array}%
\right)\left(%
\begin{array}{ccc}
(\overline{D}^F)^{-1} &  0 & 0\\
\alpha (\overline{D}^F)^{-1}   & 1  & 0 \\
P(\overline{D}^F)^{-1} & -Q\overline{\alpha}^T & D \\
\end{array}%
\right)\\
&=& \left(%
\begin{array}{ccc}
\mathbbm{1} &  0 & 0\\
-\alpha +\alpha (\overline{D}^F)^{-1}   & 1  & 0 \\
-P-Q\overline{\alpha}^T\alpha+Q\overline{\alpha}^T\alpha (\overline{D}^F)^{-1}+ D^{-1}P(\overline{D}^F)^{-1} & Q\overline{\alpha}^T-Q\overline{\alpha}^TD^{-1} & \mathbbm{1} \\
\end{array}%
\right)
\end{eqnarray*}
So, we have $\mathcal{A}_0$ and $\mathcal{A}$ are normal subgroup of $\mathcal{S}$.

Secondly, we can see that
$$X_{D, 0,0}X_{D',0,0}=X_{DD', 0,0}.$$
That means $\mathcal{D}$ is a subgroup of $\mathcal{S}$.
So, $\mathcal{S}=\mathcal{D}\rtimes\mathcal{A}$ because
$\mathcal{D}\cap \mathcal{A}=1$.
\end{proof}

We will prove that $C_{\mathcal{S}}(\mathcal{A})\leq \mathcal{A}_0$.

\begin{theorem} Let $m$ be an integer with $m\geq 2$.
Let $\mathcal{S}$ and $\mathcal{A}$ be defined as above.
Then $C_{\mathcal{S}}(\mathcal{A})\leq \mathcal{A}_0$.
\end{theorem}

\begin{proof} Let $X_{D, P', \alpha}\in C_{\mathcal{S}}(X_{\mathbbm{1}, P, 0})$, since
$X_{D, P', \alpha}=X_{D, 0, 0}X_{\mathbbm{1}, P', \alpha}$ and $[X_{\mathbbm{1}, P', \alpha}, X_{\mathbbm{1}, P, 0}]=1$,
we have
$$[X_{\mathbbm{1}, P, 0}, X_{D, 0, 0}]=1$$
by Theorem 3.7(2).
So, by Theorem 3.7$(\ast)$, we have $$D^{-1}P(\overline{D}^{-1})^F=P.~~~~~~~~~~~~~~~~~~~(\ast)$$ That means
$DP\overline{D}^F=P.$
 By the similar method of Theorem 3.4, we have $C_{\mathcal{S}}(\mathcal{A}_0)\leq \mathcal{A}$.
Hence, $C_{\mathcal{S}}(\mathcal{A})\leq C_{\mathcal{S}}(\mathcal{A}_0)\leq \mathcal{A}$.

By Theorem 3.6$(\star)$, we have $C_{\mathcal{A}}(\mathcal{A})\leq \mathcal{A}_0$ because $m\geq 2$.  Hence,  $C_{\mathcal{S}}(\mathcal{A})\leq \mathcal{A}_0$.
\end{proof}

Now, we calculate $\mathfrak{X}(\mathcal{S})$ as follows.

\begin{theorem}  Let $p$ be a prime with $p\geq 5$.
Let $\mathcal{S}$, $\mathcal{D}$, $\mathcal{A}$ and $\mathcal{A}_0$ as above.
Then $\mathfrak{X}(\mathcal{S})=\mathcal{S}$.
\end{theorem}

\begin{proof} The following proof is similar to the proof of Theorem 3.5. Let $\mathcal{N}_{ij}$ be defined as in Theorem 3.5.
Set $\widetilde{\mathcal{N}}_{ij}:=\{X_{D, P, \alpha}|D\in\mathcal{N}_{ij}, \alpha ~\mathrm{is~a}~1\times m~\mathrm{vector~and~}P \mathrm{~is}~\alpha$-$\mathrm{conjugate}$-$\mathrm{skew}$-$\mathrm{persymmetric}\}$.
We can see that
$\langle \widetilde{\mathcal{N}}_{ij}|1\leq j < i \leq m\rangle=\mathcal{S}$.
It therefore suffices to show that $1\leq \mathcal{A}\leq\widetilde{\mathcal{N}}_{ij}$ is always a $Q$-series.

We will prove that
$$[[[X_{\mathbbm{1}, P, 0}, X_{D , \hat{P}, \alpha}], X_{D', \hat{P}',\alpha'}],  X_{D^{''}, \hat{P}^{''}, \alpha^{''}}]=1. $$
for any $D, D', D^{''}\in \mathcal{N}_{ij}$, any $1\times m$ vectors $\alpha, \alpha', \alpha^{''}$,
any 0-conjugate-skew-persymmetric matrices $P$, any $\alpha$-conjugate-skew-persymmetric matrices  $\hat{P}, $   any $\alpha'$-conjugate-skew-persymmetric matrices  $\hat{P}', $  and  any $\alpha^{''}$-conjugate-skew-persymmetric matrices $\hat{P}^{''}$.

To calculate, we can set
$$D^{-1}=\mathbbm{1}+U,~ D^{'-1}=\mathbbm{1}+U',~ D^{''-1}=\mathbbm{1}+U^{''}$$
where $U, U', U^{''}$ are the elements of $M_m(\mathbb{F}_{q^2})$ and their entries satisfies
 $U_{ab} = 0, U'_{ab}=0, U^{''}_{ab}=0$ whenever
$a < i $ or $b > j$. So we can write $U, U', U^{''}$ as following block matrix:
$$U=\left(%
\begin{array}{cc}
0 & 0   \\
T & 0 \\
\end{array}%
\right),~
U'=\left(%
\begin{array}{cc}
0 & 0   \\
T' & 0 \\
\end{array}%
\right)~,
U^{''}=\left(%
\begin{array}{cc}
0 & 0   \\
T^{''} & 0 \\
\end{array}%
\right)$$
where $T$, $T'$ and $T^{''}$ are $(m-j)\times j$ matrices because $i>j$.

By Theorem 3.7$(2)$ and Theorem 3.7$(\ast)$, we have
$$[X_{\mathbbm{1}, P, 0}, X_{D , \hat{P}, \alpha}]= X_{\mathbbm{1}, P', 0} ~~\mathrm{for}~~ P'=UP+P\overline{U}^F+UP\overline{U}^F.$$
To compute $P'$, $P$ can be viewed as following block matrix:
$$P=\left(%
\begin{array}{cc}
P_1 & P_2 \\
P_3 & P_4 \\
\end{array}%
\right)$$
where $P_1$ is a $j\times (m-j)$ matrix, $P_2$ is a $j\times j$ matrix, $P_3$ is a $(m-j)\times (m-j)$ matrix, and  $P_4$ is a $(m-j)\times j$ matrix.
So, we have
$$P'=\left(%
\begin{array}{cc}
P_2\overline{T}^F & 0   \\
TP_1+P_4\overline{T}^F+TP_2\overline{T}^F & TP_2 \\
\end{array}%
\right). ~~~~~~~~~~~~~~(\ast)$$

Similarly,  we have
$$[X_{\mathbbm{1}, P', 0}, X_{D' , \hat{P}',\alpha'}]= X_{\mathbbm{1}, P'', 0} ~~\mathrm{for}~~ P{''}:=U'P'+P'\overline{U'}^F+U'P'\overline{U'}^F.$$
$P'$ can be viewed as following block matrix:
$$P'=\left(%
\begin{array}{cc}
P'_1 & P'_2 \\
P'_3 & P'_4 \\
\end{array}%
\right):=\left(%
\begin{array}{cc}
P_2\overline{T}^F & 0   \\
TP_1+P_4\overline{T}^F+TP_2\overline{T}^F & TP_2 \\
\end{array}%
\right).$$
By $(\ast)$, we have
$$P^{''}=\left(%
\begin{array}{cc}
0 & 0   \\
T'P_2\overline{T}^F+TP_2\overline{T'}^F & 0 \\
\end{array}%
\right).~~~~~~~~~~~~~(\ast\ast)$$

Similarly, we have
$$[X_{\mathbbm{1}, P^{''}, 0}, X_{D^{''} , \hat{P}^{''}, \alpha^{''}}]= X_{\mathbbm{1}, P^{'''}, 0} ~~\mathrm{for}~~ P{'''}:=U^{''}P^{''}+P^{''}\overline{U^{''}}^{F}+U^{''}P^{''}\overline{U^{''}}^{F}.$$
$P^{''}$ can be viewed as following block matrix:
$$P^{''}=\left(%
\begin{array}{cc}
P^{''}_1 & P^{''}_2 \\
P^{''}_3 & P^{''}_4 \\
\end{array}%
\right):=\left(%
\begin{array}{cc}
0 & 0   \\
T'P_2\overline{T}^F+TP_2\overline{T'}^F & 0 \\
\end{array}%
\right).$$
By $(\ast\ast)$,  we have
$$P^{'''}=\left(%
\begin{array}{cc}
0 & 0   \\
T^{''}P^{'}_2\overline{T^{'}}^{F}+T^{'}P^{'}_2\overline{T^{''}}^{F} & 0 \\
\end{array}%
\right)=0.~~~~~~~~~~~~\mathrm{For}~F'_2=0$$
That means
$$[[[X_{\mathbbm{1}, P, 0}, X_{D , \hat{P}, \alpha}], X_{D', \hat{P}',\alpha'}],  X_{D^{''}, \hat{P}^{''}, \alpha^{''}}]=X_{\mathbbm{1}, P^{'''}, 0}=X_{\mathbbm{1}, 0, 0}=1. $$
If $m=1$, it is easy to see that $\mathfrak{X}(\mathcal{S})=\mathcal{S}$ because $p\geq 5$.
So, if $m\geq 2$, that means $C_{\mathcal{S}}(\mathcal{A})\leq \mathcal{A}_0$ by Theorem 3.8. Hence, we have
$$[\Omega_1(C_{\mathcal{S}}(\mathcal{A})), \widetilde{\mathcal{N}}_{ij}; 3]\leq [\mathcal{A}_0, \widetilde{\mathcal{N}}_{ij}; 3]=1.$$
That means $1\leq \mathcal{A}\leq \widetilde{\mathcal{N}}_{ij}$ is always a $Q$-series because $p\geq 5$.
Hence, $\mathfrak{X}(\mathcal{S})=\mathcal{S}$.
\end{proof}

\textbf{ACKNOWLEDGMENTS}\hfil\break
We are very grateful for the reviewer of \cite{LXZ}. With the reviewer's suggestions and using the reviewer's ideas for symplectic groups and orthogonal
groups, we can avoid a lot of calculations for unitary groups.

\end{document}